\documentclass[12pt]{article}

\usepackage{caption}
\usepackage[top=0.8in]{geometry}
\usepackage{amsmath,amsthm,amssymb}
\usepackage{enumerate, xcolor, graphicx}
\usepackage{algorithm}
\usepackage[noend]{algpseudocode}
\usepackage{hyperref}
\usepackage{courier}
\usepackage{tcolorbox}
\usepackage{environ}
\usepackage{tikz}

\theoremstyle{theorem}
\newtheorem{theorem}{Theorem}
\newtheorem{proposition}[theorem]{Proposition}
\newtheorem{lemma}[theorem]{Lemma}
\newtheorem{corollary}[theorem]{Corollary}

\theoremstyle{definition}

\DeclareMathOperator{\E}{\mathbb{E}}

\newcounter{boxlblcounter}  

\NewEnviron{elaboration}{
\par
\begin{tikzpicture}
\node[rectangle,minimum width=1\textwidth] (m) {\begin{minipage}{1\textwidth}\BODY\end{minipage}};
\draw[dashed] (m.south west) rectangle (m.north east);
\end{tikzpicture}
}

\theoremstyle{definition}

\begin{document}

\title{The Slowest Coupon Collector's Problem}
\markright{The Slowest Coupon Collector's Problem}
\author{Tipaluck Krityakierne and Thotsaporn Aek Thanatipanonda}
\date{}
\maketitle

\begin{abstract}
In the classical coupon collector's problem, every box of breakfast cereal  contains one coupon from a collection of $n$ distinct coupons, each equally likely to appear. The goal is to find the expected number of boxes a player needs to purchase to complete the whole collection. In this work, we extend the classical problem to $k$ players who compete with one another to be the first to collect the whole collection. We find the expected numbers of boxes required for the slowest and fastest players to finish the game. The odds of a particular player being the slowest or fastest player will also be touched upon. The solutions will be discussed from both the tractable algebraic techniques as well as the probability point of views. \\

\noindent \textbf{Keywords:} coupon collector's problem; fastest player; slowest player; multiple players.
\end{abstract}

\section{Prologue}
The coupon collector's problem is a classical mathematics problem that shows up in a number of courses, from probability theory, simulations, programming, to name a few.  The classic version of the problem can be described as follows. 

{\it ``You buy cereals in order to collect coupons that come with it. The upcoming collection has $n$ collectible coupons. Each cereal box contains one coupon. Assume that every type of coupons is equally likely to appear. What is the expected number of boxes you need to buy until you have a complete set of $n$ coupons?''} 

By recalling the mean of the geometric distribution, the answer to the well-known problem above is $\E\left[X\right]=nH\left(n\right)$,  where $H\left(n\right)$ is the $n$-th harmonic number (see for example \cite[p. 225]{F}). An approximate solution to the coupon collector's problem is
\[
\E\left[X\right]\approx n\log n+\gamma n+\frac{1}{2},
\]
where $\gamma \approx 0.5772156649$ is the Euler–Mascheroni constant.

The problem has been extended to a scenario where the player has to collect multiple sets of coupons. This problem, known as the \textit{double dixie cup problem}, was solved by Newman and Shepp in 1960. Their results, published in American Mathematical Monthly,  showed that the expected number of boxes needed to complete $m$ sets of coupons is $n\left(\log n +(m-1) \log\log n+ \mathcal{O}(1)\right)$ \cite{NS}. 
In \cite{Z}, Zeilberger found the generating function for the expected 
number of types of cards of which the player has exactly $i$ copies at the end.
The note \cite{FS} gave an extensive review on approaches for solving the classical problem, and established some interesting results regarding multiple collections. Generalization of the problem to a two-player game has also been studied previously. For example, the probability that the faster player was never behind at any intermediate stage of the play has been investigated in \cite{MW}. 

In this work, we extend the problem to $k$ players who compete with one another in collecting the coupons. To our surprise,  generalizations of the coupon collector's problem in this direction seem to have never been addressed in the literature.  We thus take this opportunity to present and contribute some novel results. Notably, using algebraic recurrence relations and difference equations as the tools, our main theorem finds the expected number of boxes required for the slowest player to collect the whole collection of $n$ coupons. 
We further investigate the problem from a probability point of view, which allows us to provide full insight into the recurrence relation and the obtained solution. 

\subsection{A two-player scenario: the slower one}

As a warm up, we consider a generalized version of the expected \textit{maximum} time for two players  who are still missing $s$ and $t$ coupons, respectively. To be more precise, given  $0\leq s, t \leq n$, let $X_1(s)$ and $X_2(t)$ be random variables representing the number of boxes the first player (who are still missing $s$ coupons) and second player (missing $t$ coupons) need to open, until they each collect all $n$ coupons. 
Then,
\[
M\left(s,t\right):=\E\left[\max \{X_1{(s)},X_2{(t)}\}\right]
\]
is the expected number of boxes required for the {\it slower player}  to collect a complete set of  $n$ coupons.

Using the law of total expectation, conditioning on whether a player found a new coupon type in the next box or not, we can write a recurrence relation:
\begin{align}
\label{eq:rec_M_2players}
M\left(s,t\right)&= \left(\frac{s}{n}\right)\left(1-\frac{t}{n}\right)M\left(s-1,t\right) \,\,\,\,\,\,  {\text{first player (found a new coupon)}}\notag\\
&+ \left(1-\frac{s}{n}\right)\left(\frac{t}{n}\right)M\left(s,t-1\right) \,\,\,\,\,\,  {\text{second player}}\notag\\
&+ \left(\frac{s}{n}\cdot\frac{t}{n}\right)M\left(s-1,t-1\right) \,\,\,\,\,\,\,\,\,    {\text{both}}\notag\\
&+ \left(1-\frac{s}{n}\right)\left(1-\frac{t}{n}\right)M\left(s,t\right) \,\,\,\,\,\,  {\text{neither}}\notag\\
&+1 \,\,\,\,\,\,\,\,\,\,\,\,    {\text{(1 more box has been opened)},}
\end{align}
with the initial condition $M(0,0)=0$ and $M(s,t)=0$ if $s$ or $t < 0$.

For the two-player scenario, the recurrence relation takes a vector argument $[s, t]$. 
The initial condition $M(0,0)=0$ means that the game has ended since 
both players completed the whole set of coupons. Of course,
we will not consider the case when one of the arguments $s$ or $t$ is negative,
so we assign a zero value whenever this happens. 

\subsection{A one-player scenario}

The above recurrence can be simplified to get a recurrence for the classical one-player scenario:
\begin{align*}
M\left(s\right)&= \left(\frac{s}{n}\right)M\left(s-1\right) \,\,\,\,\,\,  {\text{the player found a new coupon}}\notag\\
&+ \left(1-\frac{s}{n}\right)M\left(s\right) \,\,\,\,\,\,  {\text{the player did not find a new coupon}}\notag\\
&+1 \,\,\,\,\,\,\,\,\,\,\,\,    {\text{(1 more box has been opened),}}
\end{align*}
with the initial condition $M(0)=0$ and $M(s)=0$ if $s<0$.
The reader can quickly verify that $M(s)=nH(s)$ indeed satisfies this recurrence together with the initial condition given, consistent with the established result of the classical coupon collector's problem. 

Our goal is to generalize the recurrence \eqref{eq:rec_M_2players} to $k>2$ players and come up with a general strategy for solving it. Before proceeding to a more detailed explanation, let us end this section with the main theorem of this paper.

\begin{theorem}
\label{thm:main}
Let $M\left(s_1,s_2,\dots,s_k\right):= E\left[\max \{X_1{(s_1)},\dots,X_k{(s_k)}\}\right]$
be the expected number of boxes required for the {\it slowest player} to collect all $n$ coupons. 
Then,
 \[ M\left(s_1,s_2,\dots,s_k\right)= nH(S) 
 -\frac{(H(S)-1)\sum_{i<j}s_is_j}{S(S-1)}
 +\mathcal{O}\left(\dfrac{1}{n}\right),\, \text{ where } S = \sum_{i=1}^ks_i. \]
\end{theorem}

In the next section, we will introduce several important tools and concepts along the way during the course of proving the theorem. 

\section{Proof of the theorem}
We first give an algebraic proof of the theorem, and then in the next section we will provide an alternative proof (from the probability view point) for the leading term, 
which gives additional insight into the recurrence and the obtained solution. 
\subsection{Recurrence relation for the slowest player} 
 
We have seen that the recurrence for two players was set up after each player has opened one more box, and checked whether or not they found a new coupon. 
Suppose now that there are $k$ players, where player $i$ is still missing $s_i$ coupons. The same idea is applied to obtain a recurrence relation for the number of boxes required for the slowest player in the $k$-player scenario. 
\begin{equation} \label{MainRec} 
M\left(s_1,s_2,\dots,s_k\right) =  \sum_{I \subseteq \{1,2,\dots,k\}} 
\underbrace{\left[\prod_{j \in I} \frac{s_j}{n}\right]
 \left[\prod_{j \not\in I} \left( 1-\frac{s_j}{n}\right) \right] M\left(V_I\right)}_{\text{players in $I$ found a new coupon}}+1, 
 \end{equation}
with the initial condition $M(0,\dots,0)=0$
and $M(s_1,s_2,\dots,s_k)=0$ if at least one of $s_i <0$.  
 
The meaning of the notation in  \eqref{MainRec} is as follows. Let $I$ be the set of index (possibly empty) of the players who found a new coupon.
For each $I$, the probability that this event happens is  
$ \left[\prod_{j \in I} \frac{s_j}{n}\right] \left[\prod_{j \not\in I} \left( 1-\frac{s_j}{n}\right) \right]. $
$V_I$ represents the updated vector argument after the players in $I$ found a new coupon, that is,
\[ V_I := [ s_1-\delta_{I}(1), s_2-\delta_{I}(2), \dots, s_k-\delta_{I}(k) ] ,\] 
where $\delta_{I}(j) =1$ if $j \in I$ and 0 otherwise. In particular, we write $V_{\{\}}$, when no players found a new coupon, and $V_{\{j\}}$ when only the player $j$ found a new coupon.

\subsection{Solutions via difference equations} 

To solve the recurrence \eqref{MainRec} for the first two leading terms, we shall reformulate the solution as a difference equation. First, we expand out \eqref{MainRec} to~get
\begin{align}  \label{ge}
 M\left(s_1,s_2,\dots,s_k\right) =& \,1+ \prod_{j=1}^k \left( 1-\frac{s_j}{n}\right)  M\left(s_1,s_2,\dots,s_k\right) \\ \notag
&+ \sum_{j=1}^k\dfrac{s_j}{n} \left[\prod_{i\neq j} 
\left( 1-\frac{s_i}{n}\right) \right] M\left(s_1,\dots,s_j-1,\dots,s_k\right)\\ \notag
&+ \sum_{i<j}\dfrac{s_is_j}{n^2} \left[\prod_{l \not \in {i,j}} 
\left( 1-\frac{s_l}{n}\right) \right] M\left(s_1,\dots, s_i-1,\dots,s_j-1,\dots,s_k\right)+\dots. \notag
\end{align}

The coefficients of this recurrence indicate that the solution must be in the form: 
\begin{equation} 
\label{eq:T}
M\left(s_1,s_2,\dots,s_k\right) =nA_1+A_0+\dfrac{1}{n}A_{-1}+\dfrac{1}{n^2}A_{-2}+\dots, 
\end{equation}
where $A_i, \;\ i \leq 1$ is a function of $s_1,s_2,\dots,s_k.$

Having  a solution written in the form \eqref{eq:T}, the proof of Theorem \ref{thm:main} boils down to the problem of identifying $A_1$ and $A_0$ (which are the coefficients of the first two leading terms of the solution).
This also explains the remainder term $\mathcal{O}\left(\frac{1}{n}\right)$ in the theorem.

\subsection{Setting up difference equations}
To set up a difference equation for $A_1$, plug the assumed form \eqref{eq:T} into \eqref{ge} and equate the resulting constant terms (which correspond to  the leading term) on both sides of \eqref{ge}:
\[
A_0 = 1+A_0-\sum_{j=1}^k \dfrac{s_j}{n} nA_1(V_{\{\}}) + \sum_{j=1}^k \dfrac{s_j}{n}nA_1(V_{\{j\}}). 
\]
(For simplicity of notation, we omit the full vector argument $V_{\{\}}=\left[s_1,s_2,\dots,s_k\right]$ of $A_i$
when there is no ambiguity.) 

Simplifying the above equation, we obtain the difference equation of $A_1$
\begin{equation} \label{Good1}
\sum_{j=1}^k s_j \left(A_1(V_{\{\}}) - A_1(V_{\{j\}}) \right) = 1,
\end{equation} 
along with the initial condition $A_1(s,0,0,\dots,0) = H(s)$, obtained from the one-player scenario.
Note that \eqref{Good1} is a first order difference equation as $V_{\{\}}=[ s_1, s_2,\dots, s_k]$ and $V_{\{j\}}=[ s_1, s_2,\dots, s_j-1,\dots,s_k]$. Although we will soon explain how we come up with the solution,
the reader may quickly check that $A_1(V)=H(S),$ where $S = \sum_{i=1}^ks_i,$ satisfies the difference equation \eqref{Good1} and the initial condition.  

Next, we find a difference equation of $A_0$. After plugging in \eqref{eq:T}, we equate the coefficients of $\dfrac{1}{n}$ (which is the second leading term) on both sides of~\eqref{ge}:
\begin{align*}
\dfrac{1}{n}A_{-1} &= -\dfrac{\sum_{j=1}^k s_j}{n}A_0+ \dfrac{1}{n}A_{-1}
+ \sum_{i<j} \dfrac{s_is_j}{n^2}nA_1 \\
&+ \dfrac{\sum_{j=1}^k s_j}{n}A_0(V_{\{j\}}) -\sum_{i<j} \dfrac{s_is_j}{n^2}nA_1(V_{\{i\}}) 
 -\sum_{i<j} \dfrac{s_is_j}{n^2}nA_1(V_{\{j\}}) +  \sum_{i<j} \dfrac{s_is_j}{n^2}nA_1(V_{\{i,j\}}).
 \end{align*}

Simplifying the above equation, we obtain the difference equation of $A_0$
\[ \sum_{j=1}^k s_j \cdot \left(A_0(V_{\{\}})-A_0(V_{\{j\}})\right) 
= \sum_{i<j} s_is_j \cdot \left( A_1(V_{\{\}})-A_1(V_{\{i\}})
 -A_1(V_{\{j\}}) + A_1(V_{\{i,j\}})  \right).
\]

We can further simplify the above equation by substituting $A_1(V) = H(S)$, and the difference equation of $A_0$  becomes
\begin{equation} \label{Good2}   
\sum_{j=1}^k s_j \cdot \left(A_0(V_{\{\}})-A_0(V_{\{j\}})\right) 
=  \sum_{i<j} s_is_j \left( \dfrac{1}{S}-\dfrac{1}{S-1} \right),
\end{equation}
together with the initial condition $A_0(s,0,0,\dots,0) = 0$. This condition is due to the absence of the other terms except the leading term, $nH(s),$ in the solution of the classical one-player scenario.  Again, \eqref{Good2}   is a first order difference equation.

\subsection{Solving difference equations}

The difference equations \eqref{Good1} and \eqref{Good2} 
that we are dealing with are a discrete version of
first order linear partial differential equations. 
In particular, both difference equations take the following form
\[ x_1 \dfrac{\partial u}{\partial x_1} +  x_2 \dfrac{\partial u}{\partial x_2} + \dots
+  x_k \dfrac{\partial u}{\partial x_k} = f(x_1,x_2,\dots,x_k), \]
in the first quadrant ($x_i > 0$).

We digress momentarily to discuss the following proposition which gives a solution to a new family of PDEs, and will be used to come up with a ``good guess'' (solution) for our difference equations. 

\begin{proposition} \label{PDE2}
Let $M, N \geq 0$ and $P$ be a multivariate polynomial where the degree of each monomial is $N$. 
Let $\displaystyle X= \sum_{j=1}^k x_j$.  Then, the solution of   
\begin{equation}
\label{eq:PDE}
\sum_{j=1}^k x_j \dfrac{\partial u}{\partial x_j} = \left(\ln{X}\right)^M\cdot\dfrac{P(x_1,x_2,\dots,x_k)}{X^N} 
\end{equation}
is 
\[
 u(x_1,x_2,\dots,x_k) =\left( \dfrac{(\ln{X})^{M+1}}{M+1}+C\right)\cdot\dfrac{P(x_1,x_2,\dots,x_k)}{X^N},
\]
for any constant $C$.
\end{proposition} 

\begin{proof} We solve this by the method of characteristics. 
Suppose $u = f(x_1,x_2,\dots,x_k)$ is a differentiable function of $x_1,\dots,x_k$, where
each $x_i$ is parameterized as a function of $t.$
Then, by the chain rule, $u$ is a differentiable function of $t$ and
\[ \dfrac{du}{dt} = \sum_{i=1}^k\dfrac{ \partial f}{ \partial x_i}\dfrac{dx_i}{dt} .\]
In order to find the solution, we solve
\[ \dfrac{ dx_i}{dt} = x_i,  \;\  1\leq i \leq k,  \;\ \;\
\text{ and }
\;\ \;\ \dfrac{ du}{dt} =  (\ln{X})^M\dfrac{P(x_1,x_2,\dots,x_k)}{X^N}. \]
Then, $x_i = c_ie^t$ where $c_i$ are constants. Substitute this into $\dfrac{du}{dt}$ to get
\[  \dfrac{ du}{dt} =  \ln{\left(\sum_ic_i e^t\right)^{M}}\dfrac{P(c_1,c_2,\dots,c_k)}{\left(\sum_i c_i\right)^N}. 
\]
The last equality holds because $P$ is a multivariate polynomial where the degree of each monomial is $N$.
The final step is to integrate both sides of the differential equation to find:
\begin{align*} 
u &=  \dfrac{P(c_1,c_2,\dots,c_k)}{(\sum_i c_i)^N}   \int \ln{\left(\sum_ic_i e^t\right)^{M}\, dt} \\
&=  \dfrac{P(c_1,c_2,\dots,c_k)}{(\sum_i c_i)^N}   \left( \dfrac{ \ln\left(\sum_ic_i e^t\right)^{M+1}}{M+1}+C\right) \\
&= \dfrac{P(x_1,x_2,\dots,x_k)}{X^N} \left( \dfrac{(\ln{X})^{M+1}}{M+1}+C \right), \,\, \text{ for any constant } C.
\end{align*}
\end{proof}

We will now make use of a more readily available solution to this family of PDEs to obtain a solution for our difference equations  \eqref{Good1} and \eqref{Good2}. 

For \eqref{Good1}, after comparing the target \eqref{Good1} to the PDE \eqref{eq:PDE}, we apply the proposition with $M=0$ and $N=0$. The obtained solution $u=\ln(X)+C$ reminds us of  the Harmonic number in a discrete version.  Thus, our guess is $\tilde{A}_1(s_1,s_2,\dots,s_k)=H(S)+C$.
Of course, we have to verify that this solution satisfies \eqref{Good1}, which obviously does. Moreover, the initial condition $A_1(0,0,\dots,0) = 0$ implies that the constant $C=0$.

Therefore, the particular solution to \eqref{Good1}, which is the coefficient of our leading term solution, is  
\[ A_1(s_1,s_2,\dots,s_k) = H(S).
\]

The target \eqref{Good2} suggests us to apply the proposition with $M=0$ and $N=2$, which gives the  solution $u = -\left(\ln{(X)}+C\right)\dfrac{\sum_{i<j}x_ix_j}{X^2}$.  Thus, a guess for the discrete analogue is
\begin{equation}
\label{eq:guess2}
\tilde{A}_0(s_1,s_2,\dots,s_k) = -\frac{\left(H(S)+C\right)\sum_{i<j}s_is_j}{S(S-1)}.
\end{equation}
(Notice the difference between the denominators $S(S-1)$ and $X^2$ first arising in the target \eqref{Good2} and the PDE \eqref{eq:PDE}, and later appearing again in their solutions.)

In order to verify that this guess is indeed the solution of \eqref{Good2}, there is one tricky calculation, which will be dealt with in the next lemma. 

\begin{lemma}  Assume $\tilde{A}_0(s_1,s_2,\dots,s_k)$ as in \eqref{eq:guess2}. Then,
\[ \sum_{j=1}^k s_j\tilde{A}_0(V_{\{j\}}) 
=  - \dfrac{H(S-1)+C}{(S-1)}\left(\sum_{i<j}s_is_j  \right)   .\]
\end{lemma}
\begin{proof}
Consider vectors $V_{\{\}}=[s_1,s_2,\dots,s_k]$ and  $V_{\{j\}}=[s'_1,s'_2,\dots,s'_k]=[s_1,s_2,\dots,s_j-1,\dots,s_k]$.
Then,
 \[\sum_{j=1}^k s_j\tilde{A}_0(V_{\{j\}}) 
= \sum_{j=1}^k s_j \left[ -\frac{H(S-1)+C}{(S-1)(S-2)}\cdot \sum_{i<l}s'_is'_l \right]
= -\frac{H(S-1)+C}{(S-1)(S-2)} \cdot \sum_{j=1}^k s_j \sum_{i<l}s'_is'_l, \]
and
\[\sum_{i<l} s'_is'_l = \sum_{i<l} s_is_l - \sum_{i=1}^k s_i +s_j.   \]
Therefore,
\[ \sum_{j=1}^k s_j \sum_{i<l} s'_is'_l =
S\sum_{i<l} s_is_l -S^2+\sum_{j=1}^k s_j^2 
= (S-2)\sum_{i<l} s_is_l,
\]
and the result is immediate.
\end{proof}

We are now ready to verify the solution of \eqref{Good2}. Substitute our guess \eqref{eq:guess2} into the l.h.s. of \eqref{Good2}, and use the lemma to obtain 
\begin{align*}
\sum_{j=1}^k s_j  \left(\tilde{A}_0(V_{\{\}})-\tilde{A}_0(V_{\{j\}})\right)
&=  -\frac{\left(H(S)+C\right)\sum_{i<j}s_is_j}{S-1} + \dfrac{\left(H(S-1)+C\right)\sum_{i<j}s_is_j}{(S-1)} \\
&=\dfrac{\sum_{i<j}s_is_j}{S-1}\left(H(S-1)-H(S)\right)
=\sum_{i<j} s_is_j \left( \dfrac{1}{S}-\dfrac{1}{S-1} \right),
\end{align*}
and so \eqref{eq:guess2} is indeed a solution of \eqref{Good2}.

The unique value of $C$ can be determined by making sure that the initial condition $A_0(1,0,\dots,0) = 0$ is satisfied. In particular,  in order to make this point a removable singularity, it is necessary
that $H(1)+C=0$, and so $C=-1$.

Thus, the particular solution of \eqref{Good2}, which is the coefficient of our second leading term solution, is
\[ A_0(s_1,s_2,\dots,s_k) = -\frac{(H(S)-1)\sum_{i<j}s_is_j}{S(S-1)}.
\]
Having obtained the closed-form formula for $A_1$ and $A_0$, Theorem \ref{thm:main} has been verified. 

The next corollary, which is an immediate corollary of Theorem \ref{thm:main}, provides the solution to our original problem when all the $k$ players start with empty hands. 

\begin{corollary}
\label{cor:main}
 \[ M\left(n,n,\dots,n\right) \approx nH(kn) 
 -\frac{(H(kn)-1)(k-1)n}{2(kn-1)}. \]
\end{corollary}

\subsection{Remarks on the remainder} 

The formula given in Theorem \ref{thm:main} becomes more precise as $n$ increases.
For example, with $n=30,$ the exact value of $M(30,30,30)$
computed numerically from \eqref{MainRec} is $151.0692567,$
while Theorem \ref{thm:main} gives $151.1009707.$ 
The general formula for $M(30,30,30)$ computed from \eqref{MainRec} with symbolic $n$~is
\begin{align*} 
M(30,30,30) &= 5.082570603n-1.376147394-\dfrac{0.9078106927}{n}-
\dfrac{1.231342610}{n^2} \\
&-\dfrac{2.159152821}{n^3}-\dfrac{4.180289796}{n^4}
-\dfrac{7.669304559}{n^5}-\dfrac{7.488122252}{n^6}\\
&+\dots  .
\end{align*}

\section{Second proof of the theorem: Insight into the leading term} 

In this section, we give an alternative proof for the leading term solution from the probability point of view.  

Consider the coupon collector's problem in a continuous-time setting. Start with one player, who is missing $s$ coupons. Through the concept of interarrival times of an inhomogeneous counting process, let $W_1\sim \text{exp}\left(\lambda_1=\dfrac{s}{n}\right)$ be the time of the first arrival of the coupon. Similarly,  let $W_i\sim \text{exp}\left(\lambda_i=\dfrac{s-i+1}{n}\right)$ be the interarrival time (elapsed time) between the $(i-1)$th and the $i$th arrivals, for $i=2,\dots,s$. It follows that the expected completion time for this particular player satisfies 
\[
E[X(s)]=E\left[\sum_{i=1}^sW_i\right]=\sum_{i=1}^s\dfrac{1}{\lambda_i}=\sum_{i=1}^s \dfrac{n}{s-i+1} = nH(s).
\]

The concept of interarrival times can be extended to find the expected maximum time for the $k$ coupon collectors' problem. Assume that player $j$ is still missing $s_j$ coupons. Let $T_1$ be the time of the first arrival of a coupon, \textit{regardless of which player finds it}. Recall a classical property that the minimum of independent exponential random variables is again exponential with the rate parameter equals to the sum of the rates. 
Then,  $T_1\sim \text{exp}\left(\lambda_1=\dfrac{S}{n}\right)$, where $S=s_1+s_2+\dots+s_k$.
In addition,  let $T_i$ be the interarrival times between the $(i-1)$th and the $i$th arrivals of the coupon, \textit{regardless of which player finds the coupon}. 
By the independence of interarrival times 
and the player who finds the coupon,  $T_i\sim \text{exp}\left(\lambda_i=\dfrac{S-i+1}{n}\right)$. Finally, the completion time of the \textit{slowest} player is simply
\begin{equation}
\label{eq:sol_cont}
E[\max\{X_1(s_1),\dots, X_k(s_k)\}]=E\left[\sum_{i=1}^ST_i\right]=\sum_{i=1}^S\dfrac{1}{\lambda_i}=\sum_{i=1}^S \dfrac{n}{S-i+1} = nH(S).
\end{equation}

Here, things simplify as two events cannot occur at the same time, and it does not matter which player finds
a next new coupon as the rate parameter of the counting process is based solely on the total number of coupons still missing at that~time.

One can write a recurrence relation for the continuous-time setting as
\begin{equation}
\label{eq:contRecurrence}
M\left(s_1,s_2,\dots,s_k\right)=\sum_{j=1}^k\left(\frac{s_j}{S}\right)M(s_1,s_2,\dots,s_j-1,\dots,s_k)+\frac{n}{S},
\end{equation}
where $S=\sum_{i=1}^ks_i$.

An interpretation of the recurrence relation is now given.
Since the rate parameter of a new arrival is $\lambda=\dfrac{S}{n}$, the last term $\dfrac{n}{S}$ represents the mean arrival time of a new coupon (regardless of which player finds it). Moreover, by recalling another classical property of the exponential distribution concerning the probability of $j$th random variable being smallest among others 
, the term $\dfrac{s_j}{S}$ is the probability that player $j$ is the one who finds the next new coupon, as one would expect. While no such explanations can be given when we solved the difference equations in the discrete-time setting, the continuous-time setting allows us to gain full insight into the recurrence relation and the solution we already obtained.  

Last but not least, the fact that the solution \eqref{eq:sol_cont} coincides with the leading term solution of the discrete-time recurrence is not a mere happenstance. In fact, the recurrence relation for the leading term solution can be obtained by dropping those terms in \eqref{MainRec} which correspond to ``multiple players finding a new coupon in the next box''. As a result, we arrive at precisely the same recurrence \eqref{eq:contRecurrence}. 

\section{Miscellaneous topics}
The final section contains a miscellaneous selection of results related to the $k$ coupon collectors' problem. 

\subsection{The fastest player}
The expected number of boxes required for the {\it fastest player} to complete the whole collection turns out to be a corollary of our main theorem.  

\begin{corollary}
The expected number of boxes required for the {\it fastest} player to collect all $n$ coupons, $\E\left[\min \{X_1{(s_1)},\dots,X_k{(s_k)}\}\right]$, is given by
\begin{align*}
&\sum_{i=1}^k M(s_i)-\sum_{i<j}M(s_i,s_j)
+\sum_{i<j<l}M(s_i,s_j,s_l)
+\cdots+(-1)^{k-1}M\left(s_1,s_2,\dots,s_k\right) \\
&= n\left(\sum_{i=1}^k H(s_i)-\sum_{i<j}H(s_i+s_j)
+\sum_{i<j<l}H(s_i+s_j+s_l)
+\cdots+(-1)^{k-1}H\left(\sum_{i=1}^{k}s_i\right)\right) +\mathcal{O}(1).
\end{align*}
\end{corollary}

\begin{proof}
Retaining only the leading term, the result follows immediately from the {\it maximum-minimum identity}
: 
\begin{align*}
\label{eq:max-min}
\min \{X_1,\dots,X_k\}&=\sum_{i=1}^k X_i-\sum_{i<j}\max\{X_i,X_j\}
+\sum_{i<j<l}\max\{X_i,X_j,X_l\}\notag \\ 
&+\cdots+(-1)^{k-1}\max \{X_1,\dots,X_k\},
\end{align*}
and the linearity of expectation. The remainder is $\mathcal{O}(1)$ as we keep only the leading term in the solution.
\end{proof}

Figure \ref{fig:combinedmaxmin} shows the graphs of the expected numbers of boxes required for the slowest player $M(n,n,\dots,n)$ and fastest player $m(n,n,\dots,n)$ to complete the whole collection of $n$ coupons, where the number of players ranges from $k=1, 2, \dots, 40$.

\begin{figure}[h]
    \centering
    \includegraphics[width=0.73\textwidth]{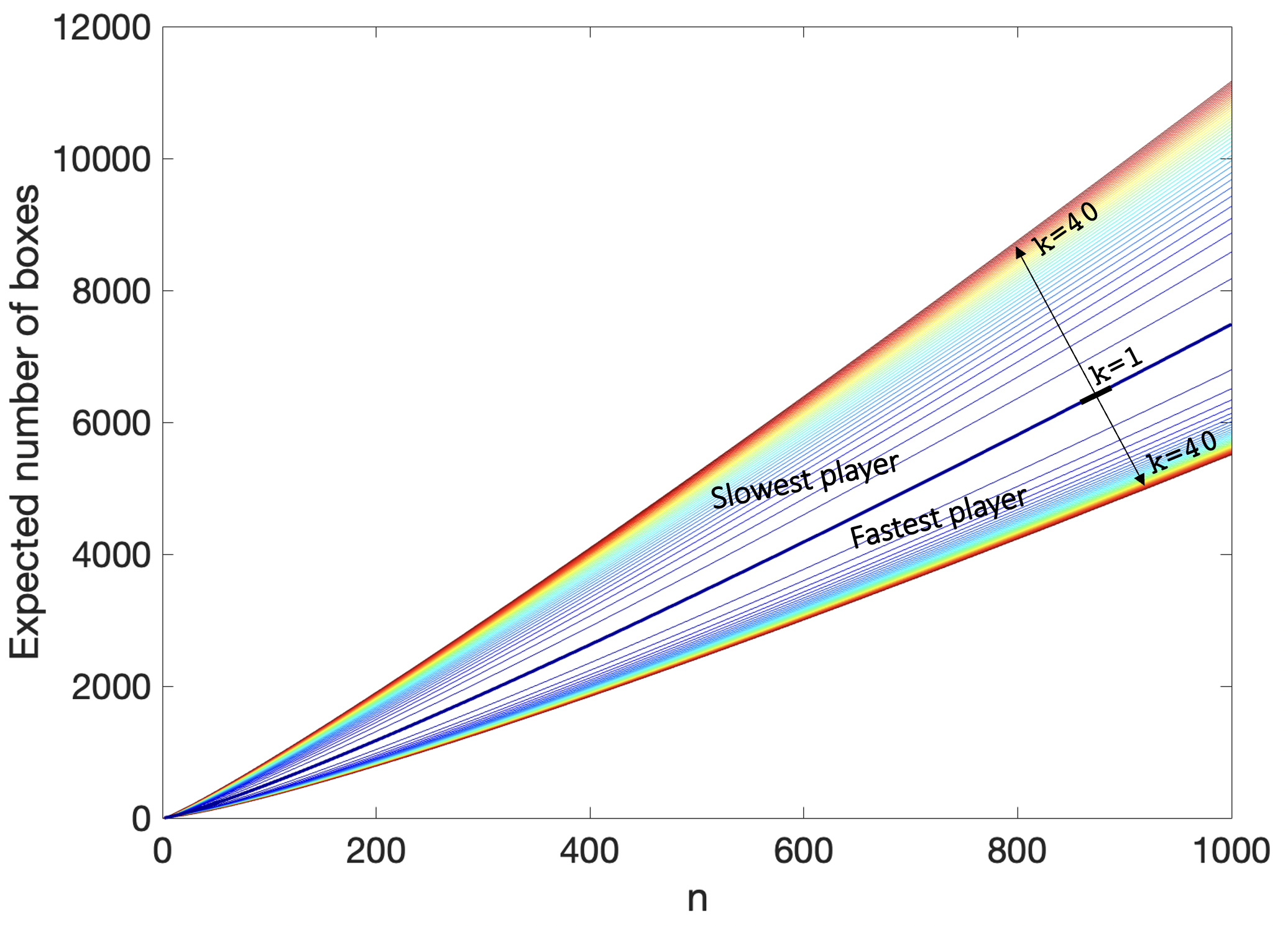}
    \caption{The expected number of boxes required for the slowest (top 40 lines) and fastest players (bottom 40 lines)  to complete the whole collection of $n$ coupons. Each line  represents a different number of $k$ players: the smallest value ($k=1$) is blue and the largest ($k=40$) red. The darkest line in the middle is a separator line corresponding to $k=1$ player, for which the slowest and fastest players are the same person.} 
    \label{fig:combinedmaxmin}
\end{figure}

\subsection{Probability of being the slowest player} 

We start with the probability of being the slowest player. 
Let $P_1\left(s_1,s_2,\dots,s_k\right)$ be the probability that the {\it first player} is the last person to complete the whole collection, i.e. $X_1\left(s_1\right)=\max \{X_1{(s_1)},X_2{(s_2)},\dots,X_k{(s_k)}\}$.  Then, we can write a recurrence
\begin{equation} \label{MainRecProb} 
P_1\left(s_1,s_2,\dots,s_k\right) =  \sum_{I \subseteq \{1,2,\dots,k\}} 
\underbrace{\left[\prod_{j \in I} \frac{s_j}{n}\right]
 \left[\prod_{j \not\in I} \left( 1-\frac{s_j}{n}\right) \right] P_1\left(V_I\right)}_{\text{players in $I$ found a new coupon}}, 
\end{equation} 
 with the initial conditions $P_1(s_1,0,\dots,0)=1$ if $s_1 \geq 0,$
 and $P_1(0,s_2,\dots,s_k)=0$ if some of $s_i > 0$,
 and $P_1(s_1,s_2,\dots,s_k)=0$ if at least one of $s_i <0$.  

The absence of  $+1$ term in this recurrence as compared to \eqref{MainRec} leads to the solution of the form: 
\begin{equation} 
\label{eq:T2}
P_1\left(s_1,s_2,\dots,s_k\right) =B_0+\dfrac{1}{n}B_{-1}+\dfrac{1}{n^2}B_{-2}+\dots, 
\end{equation}
where $B_i, \;\ i \leq 0$ is a function of $s_1,s_2,\dots,s_k.$

The next proposition finds the leading term solution $B_0$. Having  the solution written in the form \eqref{eq:T2} explains the remainder term $\mathcal{O}\left(\frac{1}{n}\right)$ in the proposition.

\begin{proposition}
\label{prop:prob_slowest}
Let $P_1\left(s_1,s_2,\dots,s_k\right)$ be the probability that the first player is slowest among the $k$ players to collect the whole set of $n$ coupons. Then, 
\[ P_1\left(s_1,s_2,\dots,s_k\right)=\frac{s_1}{\sum_{i=1}^{k}s_i} + \mathcal{O}\left(\dfrac{1}{n}\right). \]
\end{proposition}

\begin{proof}
Following the same procedure as in the proof of Theorem \ref{thm:main}, one may prove this statement by means of an algebraic recurrence relation. Nevertheless, we will alternatively prove the leading term solution using a combinatorial interpretation through the continuous-time framework.  
The number of combinations where the first player finishes last 
(i.e. the last coupon is found by the first player) is  
$\binom{s_1-1+s_2+s_3+\dots+s_k}{s_1-1,s_2,s_3,\dots,s_k}$, 
and the probability of each combination is $\dfrac{s_1!s_2!s_3!\dots s_k!}{(\sum_{i=1}^{k}s_i)!}$ 
(following from our  discussion in a continuous-time setting that $\dfrac{s_j}{\sum_{i=1}^{k}s_i}$ 
is the probability that player $j$ is the one who finds the next new coupon. 

Thus,
\begin{align*}
P\left(\text{the first player is slowest}\right) 
&=  \dfrac{s_1!s_2!s_3!\dots s_k!}{(\sum_{i=1}^{k}s_i)!} \cdot 
\binom{s_1-1+s_2+s_3+\dots+s_k}{s_1-1,s_2,s_3,\dots,s_k} \\
&=\frac{s_1}{\sum_{i=1}^{k}s_i}.
\end{align*}
This completes the proof.
\end{proof}

\subsection{Remarks on the remainder of probability} 
The formula given in Proposition \ref{prop:prob_slowest} is more precise as $n$ increases.
For example, with $n=50,$ the exact value of $P_1(24,22,14)$
computed numerically from the recurrence 
is $0.4039306738,$
while Proposition \ref{prop:prob_slowest} gives $0.40.$ 
The general formula for $P_1(24,22,14)$ computed from the recurrence with symbolic $n$ is
\begin{align*} 
P_1(24,22,14) &= 0.4+\dfrac{0.1904262018}{n}+\dfrac{0.2926113116}{n^2} +\dfrac{0.6072298683}{n^3}\\
&+\dfrac{1.461461046}{n^4}+\dfrac{3.965909505}{n^5}+\dots  .  
\end{align*}

To conclude this work, the probability of being the fastest player, whose result is a corollary to Proposition \ref{prop:prob_slowest}, will now be discussed. 

\begin{corollary}
Let $Q_{1}\left(s_1,s_2,\dots,s_k\right)$ denote the probability that the first player is the fastest player to finish, i.e. $X_1\left(s_1\right)=\min \{X_1{(s_1)},X_2{(s_2)},\dots,X_k{(s_k)}\}$.   Then,
\begin{align*}
Q_{1}\left(s_1,s_2,\dots,s_k\right)&=1-\sum_{1<i} \frac{s_1}{s_1+s_i}+\sum_{1<i<j}\frac{s_1}{s_1+s_i+s_j}\\
&-\sum_{1<i<j<l}\frac{s_1}{s_1+s_i+s_j+s_l}+\cdots+(-1)^{k-1}\frac{s_1}{s_1+\dots+s_k}+\mathcal{O}\left(\dfrac{1}{n}\right).
\end{align*}
\end{corollary}

\begin{proof}
The proof is a straightforward application of the inclusion-exclusion principle 
and Proposition \ref{prop:prob_slowest}. 
%
\end{proof}



\begin{thebibliography}{2}

\bibitem{F}  Feller, W., 1967.  {\em An introduction to probability theory and its applications.} Wiley series in probability and mathematical statistics, 3rd edn.(Wiley, New York, 1968).

\bibitem{FS} Ferrante, M. and Saltalamacchia, M., 2014. The coupon collector's problem. {\em Materials matem\`atics}, pp.1-35.

\bibitem{MW} Myers, A.N. and Wilf, H.S., 2006. Some new aspects of the coupon collector's problem. {\em SIAM review}, 48(3), pp.549-565.

\bibitem{NS} Newman, D.J. and  Shepp L., 1960. The double dixie cup problem. {\em The American Mathematical Monthly}, 67(1), pp.58-61.

%
%

\bibitem{Z} Zeilberger, D., 2001. How many singles, doubles, triples, etc. should the coupon collector expect?. {\em Unpublished manuscript available at Prof. Zeilberger’s website}.

\end{thebibliography}
\end{document}